\documentclass[a4paper,12pt]{article}

\usepackage[pdftex]{graphicx}
\usepackage{amsmath,amssymb,amsfonts,geometry}
\usepackage{comment}
\usepackage{enumerate}
\usepackage{bm}
\usepackage{ytableau}
\ytableausetup{boxsize=1em}
\usepackage{tikz}
\usetikzlibrary{positioning}

\usepackage{amsthm}
\theoremstyle{definition}
\newtheorem{Thm}{Theorem}[section]
\newtheorem*{Thm*}{定理}
\newtheorem{Def}[Thm]{Definition}
\newtheorem*{Def*}{定義}
\newtheorem{Prop}[Thm]{Proposition}
\newtheorem*{Prop*}{命題}
\newtheorem{Lem}[Thm]{Lemma}
\newtheorem*{Lem*}{補題}

\newtheorem*{Cor*}{系}
\newtheorem{Exam}[Thm]{Example}
\newtheorem*{Exam*}{例}
\newtheorem{Rem}[Thm]{Remark}
\newtheorem*{Rem*}{注意}

\newtheorem*{Exer*}{問題}

\newtheorem*{Conj*}{問題}

\numberwithin{equation}{section} 

\newcommand{\too}{\longrightarrow}
\newcommand{\N}{\mathbb{N}}

\newcommand{\R}{\mathbb{R}}
\newcommand{\C}{\mathbb{C}}

\DeclareMathOperator{\Tr}{Tr}
\DeclareMathOperator{\tr}{tr}
\DeclareMathOperator{\diag}{diag}
\providecommand{\abs}[1]{\lvert#1\rvert}

\newcommand{\E}{\mathbb{E}}
\DeclareMathOperator{\NC}{NC}
\DeclareMathOperator{\Wg}{Wg}

\DeclareMathOperator{\M}{M}
\DeclareMathOperator{\fc}{R}
\newcommand{\ncmu}{\tilde{\mu}}
\renewcommand{\S}{\mathfrak{S}}

\newcommand{\K}{K}
\newcommand{\m}{\mathfrak{m}}
\newcommand{\pp}{\mathfrak{p}} 
\renewcommand{\gg}{\widehat{\m}} 
\renewcommand{\kappa}{\widehat{\tau}}

\title{The Spectra of Principal Submatrices in Rotationally Invariant Hermitian Random Matrices and the Markov--Krein Correspondence}
\author{Katsunori Fujie and Takahiro Hasebe}
\date{\today}

\begin{document}

\maketitle


\begin{abstract}
We prove a concentration phenomenon on the empirical eigenvalue distribution (EED) of the principal submatrix 
in a random hermitian matrix whose distribution is invariant under unitary conjugacy; for example,
this class includes GUE (Gaussian Unitary Ensemble) and Wishart matrices. 
More precisely, if the EED of the whole matrix converges to some deterministic probability measure $\m$, then its fluctuation from the EED of the principal submatrix, after a rescaling, concentrates at the Rayleigh measure (in general, a Schwartz distribution) associated with $\m$ by the Markov--Krein correspondence.
For the proof, we use the moment method with Weingarten calculus and free probability.
At some stage of calculations, the proof requires a relation between the moments of the Rayleigh measure and free cumulants of $\m$. This formula is more or less known, but we provide a different proof by observing a combinatorial structure of non-crossing partitions.
\end{abstract}



\section{Introduction}
The Markov--Krein correspondence 
\begin{equation}\label{eq:MK0} 
  \int_\R \frac{1}{1-z x} \,d\m(x) = \exp\left[\int_\R\log\frac1{1-zx}\,d\tau(x)\right], \qquad z \in \C\setminus \R
\end{equation}
provides a bijection between the probability measures $\m$ on $\R$, called transition measures, and certain Schwartz distributions $\tau$. In many examples $\tau$ is a signed measure, and in such a case $\tau$ is called the Rayleigh measure of $\m$. In general, $\tau$ is the derivative (in the sense of Schwartz distribution) of a so-called Rayleigh function; see \cite{Kerov:Interlacing} for further details. The Markov--Krein correspondence appears in different contexts to describe interlacing sequences: limit shapes of large random Young diagrams \cite{LS77,VK77,Biane,Biane2}; roots of two orthogonal polynomials of large consecutive degrees \cite{Kerov:root}; eigenvalues of large random matrices and of their principal submatrices, first in the case of randomly rotated real Wigner matrices \cite{Kerov:root}, and then Wigner and Wishart matrices (without random rotation) \cite{Bufetov}.  There are also situations where the distribution $\tau$ above appears as a probability measure: Poisson--Dirichlet processes (see \cite[Section 4.1]{Kerov:Interlacing} and references therein); 
self-decomposable distributions for monotone convolution \cite{FHS20}; Harish-Chandra--Izykson--Zuber integral of rank one at a high temperature regime \cite{MP}. The reason why the same correspondence appears in different contexts seems still unclear to the authors. 

In this paper, we prove a concentration phenomenon analogous to those in \cite{Bufetov,Kerov:root} in the setting of rotationally invariant hermitian random matrices, which was posed as a conjecture in \cite{Goel-Yao}. 
Let $X_N$ be a hermitian random matrix of size $N$ whose distribution is invariant under conjugacy by unitary matrices and let $\Lambda_N =(\lambda_{1}^{(N)} \le \cdots \le \lambda_{N}^{(N)})$ be its eigenvalues. 
It is known that a diagonalization $X_N = U_N D_N {U_N}^*$ exists, where $D_N= \diag(\lambda_1^{(N)}, \lambda_2^{(N)}, \dots, \lambda_N^{(N)})$ and $U_N$ is a Haar unitary random matrix of size $N$ and independent of $D_N$ (see \cite[Proposition 6.1]{Collins-Male}).

For the principal submatrix $\tilde{X}_N$ made by removing the last row and column of $X_N$, Cauchy's interlacing law says that the eigenvalues $\tilde\Lambda_N = (\tilde\lambda_{1}^{(N)} \le \cdots \le \tilde\lambda_{N-1}^{(N)})$ of $\tilde{X}_N$ interlace with $\Lambda_N$ (see \cite[Exercise 1.3.14]{Tao}): 
\[
  \lambda_{1}^{(N)} \le \tilde\lambda_{1}^{(N)} \le \lambda_{2}^{(N)} \le \tilde\lambda_{2}^{(N)} \le \cdots \le \lambda_{N-1}^{(N)} \le \tilde\lambda_{N-1}^{(N)} \le \lambda_{N}^{(N)}.
\]
In many examples,  the empirical eigenvalue distribution $\m_N= (1/N)\sum_{i=1}^N\delta_{\lambda_i^{(N)}}$ of the random matrix $X_N$ converges, as $N\to\infty$, to a non-random probability measure, and we do assume so. Then it is not hard to see (at least with a mild assumption) that the empirical eigenvalue distribution $\tilde\m_N$ of $\tilde{X}_N$ also converges to the same limit. A question is how the difference $\m_N-\tilde\m_N$ behaves. Our main result roughly says that the rescaled difference $N(\m_N-\tilde\m_N)$, or equivalently, the Rayleigh measure
\begin{equation*}
\kappa_N = \sum_{i=1}^{N}\delta_{\lambda_{i}^{(N)}} - \sum_{j=1}^{N-1} \delta_{\tilde\lambda_{j}^{(N)}},   
\end{equation*}
 is close to the Rayleigh measure $\tau_N$ linked to the transition measure $\m_N$ by the Markov--Krein correspondence. Note that $\tau_N$ is of the form
\begin{equation*}\label{eq:tau}
\tau_N = \sum_{i=1}^{N} \delta_{\lambda_{i}^{(N)}} - \sum_{j=1}^{N-1} \delta_{\eta_{j}^{(N)}}, 
\end{equation*}
where $(\eta_{1}^{(N)} \le \cdots \le \eta_{N-1}^{(N)})$ is a sequence also interlacing with $\Lambda_N$ (see \cite[Eq.\ (2)]{Kerov:Interlacing}). 

Since our arguments are based on the moment method, we denote by $\M_k(\zeta)$ for simplicity the $k$-th moment of a measure or Schwartz distribution $\zeta$ when it is well defined.  
It should be noted here that if a probability measure $\m$ has finite moments of all orders, then $\tau$ defined via \eqref{eq:MK0} also has finite moments of all orders (see \cite[Theorem A  (d)]{AD} and \cite[Section 3.4]{Kerov:Interlacing}). Furthermore, for convenience of statements, let $\gg_N$ be the transition measure associated to the Rayleigh measure  $\kappa_N$; then the main result can alternatively be phrased that $\gg_N$ is close to $\m_N$.  

The precise statement of the main result is as follows, which answers to a conjecture announced by Goel and Yao \cite{Goel-Yao}. 
\begin{Thm}\label{main}
Let $\m_N,\tau_N,\gg_N,\kappa_N$ be as above, $\m$ be a probability measure on $\R$ and $\tau$ be related to $\m$ via \eqref{eq:MK0}. Assume that
\begin{equation}\label{eq:ass1}
\sup_{N\ge1} \E[\M_k(\m_N)] <\infty \quad \text{and}\quad \M_k(\m)<\infty, \qquad  k \in 2\N.   
\end{equation}
and $\m_N$ converges in moments to $\m$ in probability:    
\begin{equation}\label{eq:weak_moment1}
\lim_{N\to\infty}\mathbb P [\abs{\M_k(\m_N) - \M_k(\m)} \ge \epsilon] =0, \qquad  k \in \N,~  \epsilon>0.  
\end{equation} 
Then we have
\begin{equation*}
\lim_{N\to\infty}\|\M_k(\kappa_N) - \M_k(\tau)\|_{L^2} =0, \qquad  k \in \N, 
\end{equation*}
and 
\begin{equation*}
\lim_{N\to\infty}\mathbb P [\abs{\M_k(\gg_N) - \M_k(\m)}\ge \epsilon]=0, \qquad  k \in \N,~  \epsilon>0.
\end{equation*}
In particular, if the moment problem for $\{\M_k(\m)\}_{k\ge1}$ is determinate then $\gg_N$ weakly converges to $\m$ in probability: 
\begin{equation*}
\lim_{N\to\infty}\mathbb P \left[\left|\int_\R f(x)\, \gg_N(dx) -\int_\R f(x)\, \m(dx) \right| \ge \epsilon\right ]=0, \qquad  f \in C_b(\R), ~  \epsilon>0. 
\end{equation*}
\end{Thm}
\begin{Rem}
\begin{enumerate}[\rm(i)] 
\item Since the relation between the moments $\{ \M_n(\m) \}_{n \in \N}$ (resp.\ $\{ \M_n(\m_N) \}_{n \in \N}$) and $\{ \M_k(\tau) \}_{k \in \N}$ (resp.\ $\{ \M_k(\tau_N) \}_{k \in \N}$) is the same as that between complete symmetric functions and Newton power sums (see \eqref{eq:power&comp} below), the convergence \eqref{eq:weak_moment1} holds if and only if $\tau_N$ converges in moments to $\tau$ in probability.
\item \eqref{eq:ass1} and \eqref{eq:weak_moment1} imply the convergence of moments in $L^p$ norm for every $p \in [1,\infty)$; see Proposition \ref{prop:Lp}. 

\item The assumptions  \eqref{eq:ass1} and \eqref{eq:weak_moment1} are satisfied by appropriately  normalized Gaussian Unitary Ensemble (GUE) \cite[Theorem 4.1.5]{HP}, where $\m$ is Wigner's semicircle law $(1/(2\pi))\sqrt{4-x^2}\,dx$.  For GUE (actually, more general Wigner matrices), a finer result on the fluctuation of $\kappa_N$ from $\tau$ is also known in \cite{ES2017} stated in the language of rectangular Young diagrams; see also \cite{Sodin2017}.

\end{enumerate}
\end{Rem}

The proof is based on Weingarten calculus and free probability which allow us to compute the moments of the principal submatrix: 
\begin{equation}\label{eq:key_moments}
\E \circ \Tr[(\tilde{X}_N)^{k}] = \E\circ \Tr[D_NU_NP_N{U_N}^* D_NU_N P_N{U_N}^* \cdots D_NU_N P_N{U_N}^*],    
\end{equation}
where $ P_N=\diag(1,1,\dots,1,0)$. 

In fact, the joint distribution of $(\tilde\lambda_{1}^{(N)} \le \cdots \le \tilde\lambda_{N-1}^{(N)})$ is explicit under the condition that $(\lambda_{1}^{(N)} \le \cdots \le \lambda_{N}^{(N)})$ a constant sequence; it is proportional to the Vandermond determinant  \cite[Proposition 4.2]{Ba2001} (see also the expository paper \cite{Faraut}). Using this explicit formula might be an alternative approach for computing \eqref{eq:key_moments} and hence for a proof of Theorem \ref{main}; however, the authors are not sure whether this direction is promising.  

At some stage of calculations of \eqref{eq:key_moments} with Weingarten calculus, it turns out that the following formula \eqref{eq:tau&fcmu} is crucial.  

 \begin{Thm}\label{main2} 
 Suppose that $\m$ is a probability measure on $\R$ with finite moments of all orders and $\tau$ be defined via \eqref{eq:MK0}. Then the formula 
   \begin{equation}\label{eq:tau&fcmu}
     \M_k(\tau) = \sum_{\rho \in \NC(k)} (k+1-\abs{\rho})\fc_\rho(\m)
   \end{equation}
  holds for every $k \in \N$, where $\NC(k)$ is the set of non-crossing partitions of $\{1, \dots, k \}$ and $\fc_\rho(\m)$ is the free cumulant of $\m$.
 \end{Thm}
This formula gives an explicit combinatorial relation between two bases in the Kerov-Olshanski algebra: the moments of $\tau$ and free cumulants of $\m$.  It can be easily proved by combining known formulas for complete symmetric functions as follows. The moments of $\tau$ and the free cumulants of $\m$ can be identified with the elements $\{p_n(A)\}_{n\ge1}$ and $\{(-1)^ne_n^*(A)\}_{n\ge1}$ in \cite{Lassalle}, respectively; the latter fact is noted on page 2242 of \cite{Lassalle}. Combining (4.5) and the formula right before (4.10) in \cite{Lassalle} allows one to express $\{p_n^*(A)\}_{n\ge1}$ in terms of $\{e_n(A)\}_{n\ge1}$ as a sum over integer partitions. Applying the involution gives a formula that expresses $\{p_n(A)\}_{n\ge1}$ in terms of $\{(-1)^ne_n^*(A)\}_{n\ge1}$. This formula can be transformed into the sum over non-crossing partitions via \cite[Corollary 9.12]{Nica-Speicher}, which amounts to Theorem \ref{main2}.  

In this paper, we provide a different proof of Theorem \ref{main2} based on non-crossing partitions. A key observation is that the coefficient $k+1-|\rho|$ coincides with the cardinality of the Kreweras complement $K(\rho)$. Since the coefficients are so simple, one may expect that there is a combinatorial structure behind. Indeed, we will introduce the notion of ``Kreweras decomposition'' of a non-crossing partition and count the number of such decompositions as a crucial ingredient of the proof.

After this introduction, this paper is structured as follows. Section \ref{sec2} consists of preliminaries on Weingarten calculus, free probability and symmetric groups. Section \ref{sec3} provides the proof of Theorem \ref{main}, as well as an alternative proof of Theorem \ref{main2} as mentioned.  Some results on convergence of random measures are proved in Appendix \ref{App1}.


\section{Preliminaries}\label{sec2}
In this section, we introduce standard notions in free probability and related fields for later use in the proof of the main results.

\subsection{Weingarten calculus}\label{sec:WG}
Computation of mixed moments of Haar unitary random matrices $U_N$ and deterministic matrices is called Weingarten calculus. 
For $\sigma \in \S_k$, let $\Tr_\sigma[A_1,A_2,\dots, A_k]$ be the product of traces according to the cycle decomposition of $\sigma$; for example if $\sigma =(1,3,2,5)(4)(6,9)(7,8)$ then $\Tr_\sigma[A_1,A_2,\dots, A_9] = \Tr(A_1A_3A_2A_5)\Tr(A_4)\Tr(A_6A_9)\Tr(A_7A_8)$. Similarly, for a sequence $\{\alpha_n\}_{n\ge1} \subset \C$ we define $\alpha_\sigma$ to be the product of $\alpha_n$'s according to the sizes of cycles; in the above example, $\alpha_\sigma=\alpha_4\alpha_1{\alpha_2}^2$. 

  Let $A_i, B_i \; (i= 1, \dots, k)$ be $N \times N$ matrices.
  Then
  \begin{align*}
    &\E \circ \Tr_\sigma[A_1 U_N B_1 {U_N}^*, \dots, A_k U_N B_k {U_N}^*] \\
    &\qquad = \sum_{\substack{\sigma_1, \sigma_2, \sigma_3 \in \S_k \\ \sigma_1 \sigma_2 \sigma_3 = \sigma}}
    \Tr_{\sigma_1}[A_1, \dots, A_k] \Tr_{\sigma_2}[B_1, \dots, B_k] \Wg(\sigma_3, N)
  \end{align*}
  for all $\sigma \in \S_k$.  In particular, in the case of $\sigma= \gamma_k = (1, 2, \dots, k)$ the above formula specializes to
  \begin{align}
   & \E \circ \Tr[(A_1 U_N B_1 {U_N}^*) \cdots (A_k U_N B_k {U_N}^*)] \notag \\
   &\qquad  = \sum_{\sigma, \pi \in \S_k}
    \Tr_{\sigma}[A_1, \dots, A_k] \Tr_{\pi}[B_1, \dots, B_k] \Wg(\pi^{-1}\sigma^{-1}\gamma_k, N),  \label{eq:WG_formula}
  \end{align}
 see \cite[Proposition 2.3]{Collins-Sniady}.
The coefficients $\Wg(\sigma, N)$ are called the Weingarten function. Its asymptotic behavior for large $N$ is known in the form
\begin{equation}\label{eq:asymptotic_Weingarten}
    N^{k+\abs{\sigma}} \Wg(\sigma,N) = \mu_k(\sigma) + O \left( \frac{1}{N^2} \right), \qquad \sigma \in \S_k.  
\end{equation}
The number $\abs{\sigma}$, called the length function, is the minimal number $l$ for which $\sigma$ can be written as a product of $l$ transpositions, and the number $\mu_k(\sigma)$ above is expressed in terms of the Catalan numbers $C_n = (2n)!/(n!(n+1)!)$ as 
\begin{equation}\label{eq:moebius_s}
\mu_k(\sigma) = \prod_{1 \le j \le l}(-1)^{\abs{\pi_j}}C_{\abs{\pi_j}} 
\end{equation}
where  $\sigma = \pi_1 \cdots \pi_l$ is the cycle decomposition of $\sigma$;  see \cite[Theorem 2.7]{Collins-Matsumoto}

\subsection{Free cumulants}\label{sec:free}
This section summarizes notations and facts on free cumulants. The reader is referred to \cite{Nica-Speicher} for further details. Let $L$ be a finite linearly ordered set. A partition of $L$ is the collection of nonempty disjoint subsets of $L$ whose union is $L$. For a partition $\rho =\{B_1,B_2, ..., B_r\}$ of $L$, each element $B_i$ is called a block and the cardinality $r$ is denoted by $|\rho|$.
It is referred to as a crossing partition if there are two blocks $B_i, B_j~(i\ne j)$ and elements $a,b \in B_i, c,d \in B_j$ such that $a<c<b<d$. Otherwise, it is called a non-crossing partition. The set of the non-crossing partitions of $L$ will be denoted by $\NC(L)$.  In particular, when $L=[k]:=\{1,2,\dots,k\}$, $\NC([k])$ is simply denoted by $\NC(k)$.

For $\nu,\rho \in \NC(k)$, the notation $\nu\le \rho$ means that for every $B \in \nu$ there exists $C \in \rho$ such that $B\subseteq C$. This defines a poset structure on $\NC(k)$.  The maximum element regarding this partial order is $1_k := \{[k]\}$, the partition consisting of one block $[k]$, and the minimum is $0_k:=\{\{1\},\{2\},\dots, \{k\}\}$.

The Kreweras complement of a non-crossing partition $\rho \in \NC(k)$ is defined as follows. Inserting additional points $[\overline{k}]:=\{\overline{1}, \overline{2}, \dots, \overline{k}\}$ to $[k]$, suppose that  $L_k=\{1,\overline{1}, 2,\overline{2}, \dots, k,\overline{k}\}$ is a linearly ordered set with the order as displayed.  Take the maximal non-crossing partition $\nu$ of $[\overline{k}]$ such that $\rho \cup \nu \in \NC(L_k)$. Then deleting bars over the integers, $\nu$ is called the Kreweras complement of $\rho$ and denoted by $K(\rho)$. For convenience,  we sometimes keep the bars and regard $K(\rho)$ as a non-crossing partition on $[\overline{k}]$.

\begin{Exam}
If $\rho = \{\{1,7\}, \{2,5,6\}, \{3\}, \{4\}, \{8,9\} \}$ then the following picture
 
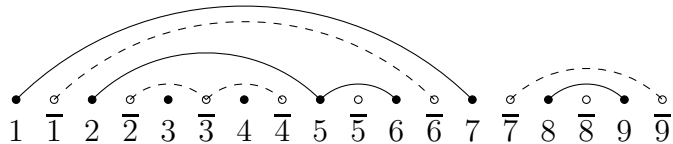
\begin{center}
      \begin{tikzpicture}[scale=0.5]
        \foreach \x in {1,...,9}{
          \draw[circle,fill] (2*\x-1,0)circle[radius=1mm];
               \node at (2*\x-1,-0.82) {$\x$};
        }
        \foreach \x in {1,...,9}{
          \draw[circle] (2*\x,0)circle[radius=1mm];
               \node at (2*\x,-0.73) {$\overline{\x}$};
        }
        \foreach \x/\y in {1/7,2/5,5/6,8/9} {
           \draw(2*\x-1,0) to[bend left=45] (2*\y-1,0);
        }
        \foreach \x/\y in {1/6,2/3,3/4,7/9} {
           \draw[dashed](2*\x,0) to[bend left=45] (2*\y,0);
        }
      \end{tikzpicture}
      \end{center}
  shows that $K(\rho) = \{\{1,6\}, \{2,3,4\},\{5\},\{7,9\},\{8\}\}$. 
\end{Exam}

For a sequence $\{\alpha_n\}_{n \in \N} \subset \C$ and a partition $\rho$ of $[k]$, define
\begin{equation}
\alpha_\rho := \prod_{B \in \rho} \alpha_{|B|}.
\end{equation}

For a probability measure $\m$ on $\R$ with finite moments of all orders, the free cumulants $\{ \fc_k(\m) \}_{k \in \N}$ of $\m$ are determined recursively by the moment--cumulant formula
\begin{equation}\label{eq:mc}
\M_\nu(\m) = \sum_{\substack{\rho \in \NC(k)\\ \rho \le \nu}} \fc_\rho(\m), \qquad \nu \in \NC(k), \quad k\in \N.
\end{equation}
Actually, it suffices to take $\nu=1_k, k=1,2,3,\dots$ to determine the free cumulants, and then the above formula can be proved for all $\nu \in \NC(k), k=1,2,3,\dots$. More explicitly, free cumulants can be expressed as
\begin{equation}\label{eq:cm}
\fc_\rho (\m) =  \sum_{\substack{\nu \in \NC(k) \\ \nu \le \rho}} \M_\nu(\m) \ncmu_{k}(\nu, \rho),  \qquad \rho \in \NC(k), \quad k\in\N,
\end{equation}
where $\ncmu_{k}$ is the M\"{o}bius function on the poset $\NC(k)$.

\subsection{Non-crossing partitions and symmetric groups} \label{subsec:Isomorph}

The set of non-crossing partitions can be embedded into the symmetric group. Here we collect needed facts. For further details, the reader is referred to \cite{Nica-Speicher}.

The length function (see Section \ref{sec:WG}) on symmetric groups satisfies the following properties: for all $\sigma, \pi \in \S_k$,
\begin{align}
      &\abs{\pi \sigma \pi^{-1}} = \abs{\sigma}, \\
      &\abs{\sigma \pi} \le \abs{\sigma} + \abs{\pi},   \label{eq:triangular}\\
      &\abs{\sigma \pi} \equiv \abs{\sigma} + \abs{\pi} \pmod 2. \label{eq:mod2}
  \end{align}
The number $\#(\sigma)$ of cycles in the cycle decomposition of $\sigma$ is known to satisfy
  \[
    \#(\sigma) + \abs{\sigma} = k.
  \]

Let $d$ be a metric on $\S_k$ defined by $d(\sigma, \pi) = \abs{\sigma^{-1} \pi}$. The geodesic set from the unit $e$ to $\gamma_k := (1, \dots, k)$ is defined by
\[
        \S_{\NC}(\gamma_k) = \{ \, \sigma \in \S_k \mid d(e, \sigma) + d(\sigma, \gamma_k) = d(e, \gamma_k) \; (= k-1) \,  \}
\]
  For $\sigma, \pi \in \S_{\NC}(\gamma_k)$, denote by $\sigma \le \pi$ if $\sigma$ and $\pi$ are on a common geodesic and $d(e,\sigma) \le d(e,\pi)$, namely, if $d(e, \sigma) + d(\sigma, \pi) = d(e, \pi)$, or equivalently,  $\abs{\sigma} + \abs{\sigma^{-1}\pi} = \abs{\pi}$.

For a partition $\rho\in \NC(k)$, each block $B=\{i_1,i_2,\dots, i_p\} \in\rho$ whose elements are arranged in the increasing order associates the cyclic permutation $\pi=(i_1, i_2,\dots, i_p)$, so that $\rho$ associates the permutation $\mathcal{P}_\rho := \pi_1 \pi_2 \cdots \pi_l$, where $l=\#(\mathcal{P}_\rho) = \abs{\rho}$.
This embedding becomes a poset isomorphism
  \[
    \mathcal{P} \colon \NC(k) \too \S_{\NC}(\gamma_k),
  \]
 see \cite[Proposition 23.23]{Nica-Speicher}. We need the following facts later: for $\sigma=\mathcal{P}_\nu, \pi=\mathcal{P}_\rho\in  \S_{\NC}(\gamma_k)$,
\begin{enumerate}[\rm(i)]
\item the relation $\nu \le \rho$ holds in $\NC(k)$ if and only if $\abs{\sigma} + \abs{\sigma^{-1}\pi} + \abs{\pi^{-1} \gamma_k} = k-1$,
\item  $\mu_k(\sigma^{-1}\pi) = \ncmu_k(\nu,\rho)$,
\item $\pi^{-1} \gamma_k = \mathcal{P}_{K(\rho)}$; in particular $\#(\pi^{-1}\gamma_k) = \abs{K(\rho)}$.
\end{enumerate}

Similar results hold for ${\gamma_k}^{(1)}{\gamma_k}^{(2)} \in \S_{2k}$ instead of $\gamma_{2k}$, where
\[
  {\gamma_k}^{(1)} = (1, \dots, k)(k+1) \cdots (2k) \qquad \text{and}\qquad {\gamma_k}^{(2)} = (1) \cdots (k) (k+1, \dots, 2k).
\]
Correspondingly, let
\[
  (1_k^{(1)},1_k^{(2)}) := \{ \{1, \dots, k\}, \{k+1, \dots, 2k\}\},
\]
then $\mathcal{P}_{(1_k^{(1)},1_k^{(2)})} = {\gamma_k}^{(1)}{\gamma_k}^{(2)}$.
Via the imbedding
\[
  \NC(k) \times \NC(k) \cong [0_{2k}, (1_k^{(1)},1_k^{(2)})] \subset \NC(2k)
\]
the restriction of the mapping $\mathcal P$ induces an isomorphism between $\NC(k) \times \NC(k)$ and
\begin{equation} \label{isomoph_part}
  \S_{\NC}({\gamma_k}^{(1)}{\gamma_k}^{(2)}) = \{ \, \sigma \in \S_{2k} \mid d(e, \sigma) + d(\sigma, {\gamma_k}^{(1)}{\gamma_k}^{(2)}) = d(e, {\gamma_k}^{(1)}{\gamma_k}^{(2)}) \; (= 2k-2) \,  \}.
\end{equation}
For $(\rho_1, \rho_2) \in \NC(k) \times \NC(k)$ and $\pi = \mathcal{P}_{(\rho_1,\rho_2)} \in \S_{\NC}({\gamma_k}^{(1)}{\gamma_k}^{(2)})$, the element $\pi^{-1}{\gamma_k}^{(1)}{\gamma_k}^{(2)}$ corresponds to  $(K(\rho_1), K(\rho_2))$ under the isomorphism $\mathcal P$, and in particular $\#(\pi^{-1}{\gamma_k}^{(1)}{\gamma_k}^{(2)}) = \abs{K(\rho_1)} + \abs{K(\rho_2)} = \abs{K(\rho)} + 1$, where $\rho = (\rho_1, \rho_2)$ is regarded as a partition in $\NC(2k)$.
Note that this relation can be clearly understood in terms of the relative Kreweras complement; however, we will not use this technical notion since it is not directly needed in this paper.


\section{Proof of the main results}\label{sec3}

\subsection{Proof of Theorem \ref{main2}} \label{subsec:submain}

Let us start to prove the combinatorial formula \eqref{eq:tau&fcmu} by induction on the degree $k$. In this subsection, we keep the assumptions and notation in Theorem \ref{main2}. 
To begin, the original formula for the Markov--Krein correspondence \eqref{eq:MK0} implies the recursive relation
\begin{equation}
  \M_k(\tau) = k\M_k(\m) - \sum_{r=1}^{k-1} \M_r(\tau) \M_{k-r}(\m), \qquad k \in \N, \label{eq:power&comp}
\end{equation}
which is exactly the relation satisfied by complete symmetric functions and Newton power sums \cite[(3.2.4) and Section 3.4]{Kerov:Interlacing}.

Thanks to the moment-cumulant formula \eqref{eq:mc}, the RHS of the desired formula (\ref{eq:tau&fcmu}) may be transformed into
\[
\begin{split}
  \sum_{\rho \in \NC(k)} (k+1-\abs{\rho})\fc_\rho(\m) &= k \sum_{\rho \in \NC(k)} \fc_\rho(\m) - \sum_{\rho \in \NC(k)} (\abs{\rho}-1)\fc_\rho(\m) \\
  &= k \M_k(\mu) - \sum_{\rho \in \NC(k)} (\abs{\rho}-1)\fc_\rho(\m).
\end{split}
\]
Hence, according to the recursive equation \eqref{eq:power&comp},  formula (\ref{eq:tau&fcmu}) is eventually equivalent to
\begin{equation} \label{eq:equi1}
  \sum_{\rho \in \NC(k)} (\abs{\rho}-1)\fc_\rho(\m) = \sum_{r=1}^{k-1} \M_r(\tau) \M_{k-r}(\m).
\end{equation}
 By the induction hypothesis up to the degree $k-1$ and the moment-cumulant formula, the RHS of (\ref{eq:equi1}) can be written as
  \begin{equation}\label{eq:sum}
     \sum_{r=1}^{k-1} \sum_{\substack{\overline{\rho} \in \NC(r) \\  \underline{\rho} \in \NC(k-r)}} |K( \overline{\rho})| \fc_{ \overline{\rho}}(\m) \fc_{ \underline{\rho}}(\m).
\end{equation}
The cardinality  $|K( \overline{\rho})|$ can be interpreted as the number of inserting $ \underline{\rho}$ into $\overline{\rho}$ in the following way:
\begin{enumerate}[\rm(P1)]
\item\label{P1} pick $r \in \{1,2,\dots,k-1\}$, $\overline{\rho} \in \NC(r)$ and $\underline{\rho}\in \NC(k-r)$;
\item\label{P2}  pick a block $B$ of $\K(\overline{\rho})$, where $K(\overline{\rho})$ is interpreted as a partition on the points $[\overline{r}]$ interlacing with $[r]$;
 \item\label{P3} substitute the partition $\underline{\rho}$ into the last point of $B$.
\end{enumerate}

The steps (P\ref{P2}) and (P\ref{P3}) provide a way to insert $\underline{\rho}$ into $\overline{\rho}$, which yields a non-crossing partition $\rho \in \NC(k)$; see also Example \ref{ex:insert}.  The sum \eqref{eq:sum} can then be expressed as
 \begin{equation}\label{eq:sum2}
 \sum_{\rho} \fc_\rho(\m),
\end{equation}
where $\rho$ runs over all the non-crossing partitions appearing as a result of (P\ref{P1})--(P\ref{P3}). Note that the same non-crossing partition $\rho$ may appear more than once, and the sum \eqref{eq:sum2} needs to count the multiplicity.
Actually,  in order to have (\ref{eq:equi1}),  we need to demonstrate that each $\rho \in \NC(k)$ appears exactly $\abs{\rho}-1$ times. To achieve this, we introduce the notion of Kreweras decomposition of a non-crossing partition, which describes the relation between $\rho, \overline{\rho}$ and $\underline{\rho}$ above.

\begin{Exam}\label{ex:insert}
For the non-crossing partitions $\overline{\rho} = \{\{1,7\},\{2,5,6\},\{3\},\{4\},\{8,9\}\}$ and $\underline{\rho}=\{\{1,3\},\{2\}\}$,
the Kreweras complement $K(\overline{\rho})$ is the partition described by the dashed curves below
 
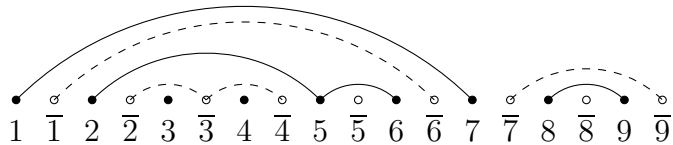
\begin{center}
      \begin{tikzpicture}[scale=0.5]
        \foreach \x in {1,...,9}{
          \draw[circle,fill] (2*\x-1,0)circle[radius=1mm];
               \node at (2*\x-1,-0.82) {$\x$};
        }
        \foreach \x in {1,...,9}{
          \draw[circle] (2*\x,0)circle[radius=1mm];
               \node at (2*\x,-0.73) {$\overline{\x}$};
        }
        \foreach \x/\y in {1/7,2/5,5/6,8/9} {
           \draw(2*\x-1,0) to[bend left=45] (2*\y-1,0);
        }
        \foreach \x/\y in {1/6,2/3,3/4,7/9} {
           \draw[dashed](2*\x,0) to[bend left=45] (2*\y,0);
        }
      \end{tikzpicture}
      \end{center}
and hence the Kreweras complement has the blocks $\{\overline{1}, \overline{6}\}, \{\overline{2},\overline{3},\overline{4}\}, \{\overline{5}\}, \{\overline{7},\overline{9}\}, \{\overline{8}\}$. According to (P\ref{P3}) we are allowed to place $\underline{\rho}$ at any point of $\{\overline{6}, \overline{4}, \overline{5}, \overline{8}, \overline{9}\}$. For example, if we choose $\overline{4}$ then the resulting non-crossing partition $\rho$ is
  
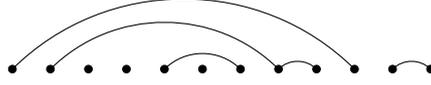
\begin{center}
      \begin{tikzpicture}[scale=0.5]
        \foreach \x in {1,...,12}{
          \draw[circle,fill] (\x,0)circle[radius=1mm];
        }
        \foreach \x/\y in {1/10,2/8,8/9,11/12, 5/7} {
           \draw(\x,0) to[bend left=45] (\y,0);
        }
     \end{tikzpicture}
      \end{center}
\end{Exam}

\begin{Def}
\begin{enumerate}[(1)]
\item  For $\nu \in \NC(r)$, a Kreweras point of $\nu$ is the last point of a block of the Kreweras complement $\K(\nu)$ regarded as a partition on $[\overline{r}]$ that interlaces with $[r]$.

\item
  For $\rho \in \NC(k)$, a pair $(\overline{\rho}, \underline{\rho})$ of nonempty disjoint subsets of $\rho$ such that $\overline{\rho} \cup \underline{\rho} = \rho$ and the union of all elements of $\underline{\rho}$ is an interval of $[k]$, that is, there exist some $i<j$ such that
  \[
  \bigcup_{V \in \underline{\rho}} V = [i,j] =\{i, i+1, \dots, j \}.
  \]
  If the position of $\underline{\rho}$ is a Kreweras point of $\overline{\rho}$, then we call $(\overline{\rho}, \underline{\rho})$ a Kreweras decomposition of $\rho$, $\overline{\rho}$ an outer partition of $\rho$ and $\underline{\rho}$ an inner partition of $\rho$.
  \end{enumerate}
\end{Def}

\begin{Exam}
The non-crossing partition $\rho =\{\{1,8\}, \{2,3\}, \{4,6,7\},\{5\},\{9,10\}\}$ can be described as 
 
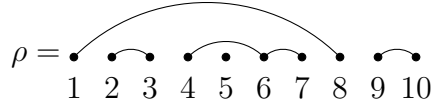
\begin{center}
      \begin{tikzpicture}[scale=0.5]
      \node at (0,0) {$\rho=$};
        \foreach \x in {1,...,10}{
           \draw[circle,fill] (\x,0)circle[radius=1mm];
           \node at (\x,-0.8) {$\x$};
        }
        \foreach \x/\y in {1/8,2/3, 4/6,6/7, 9/10}{
           \draw(\x,0) to[bend left=45] (\y,0);
        }
      \end{tikzpicture}
\end{center}
and it has the four inner partitions $\rho_1=\{ \{2,3\}, \{4,6,7\},\{5\}\}, \rho_2=\{ \{4,6,7\},\{5\}\} , \rho_3=\{\{5\}\},\rho_4=\{\{9,10\}\}$.  
Any other subsets of $\rho$ are not inner partitions; for example, $\rho' =\{\{2,3\}\}$ has the support $\{2,3\}$ of interval form, but the Kreweras complement of $\rho\setminus \rho'$ is described by the dashed curves and white singletons in the picture
  
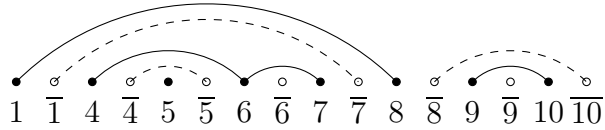
\begin{center}
      \begin{tikzpicture}[scale=0.5]
        \foreach \x in {1,...,8}{
           \draw[circle,fill] (2*\x-1,0)circle[radius=1mm];
        }
        \foreach \x/\y in {1/6,2/4,4/5, 7/8}{
           \draw(2*\x-1,0) to[bend left=45] (2*\y-1,0);
        }
          \foreach \x in {1,...,8}{
           \draw[circle] (2*\x,0)circle[radius=1mm];
        }
        \foreach \x/\y in {1/5,2/3,6/8}{
           \draw[dashed](2*\x,0) to[bend left=45] (2*\y,0);
        }
        \node at (1,-0.8) {$1$};
          \node at (2,-0.73) {$\overline{1}$};
          \foreach \x in {4,...,10}{
          \node at (2*\x-5,-0.8) {$\x$};
        }
       \foreach \x in {4,...,10}{
          \node at (2*\x-4,-0.73) {$\overline{\x}$};
        }
      \end{tikzpicture}
\end{center}
so that the position of the removed block $\{2,3\}$ was at the point $\overline{1}$, which was not the last point of the block $\{\overline{1},\overline{7}\}$.

  \end{Exam}

The goal is then to demonstrate that each $\rho \in \NC(k)$ has exactly $\abs{\rho}-1$ Kreweras decompositions. The proof is based on induction, which depends on the following nesting structure of inner partitions.

\begin{Lem} \label{lem:inner_partition}
  Suppose that $\rho \in \NC(k)$ and its first block which contains $1$ divides $[k]$ into (non-empty) $l$ segments $I_1, \dots, I_l$.
  Then $\rho_j := \rho \! \mid_{I_j}$ is an inner partition of $\rho$ for every $j$, and moreover, every inner partition of $\rho_j$ is an inner partition of $\rho$.
  Conversely, any inner partition of $\rho$ is some $\rho_j$ or its inner partition.
\end{Lem}

\begin{proof}
  It is clear that all $\rho_j \; (j=1, \dots, l)$ are inner partitions of $\rho$.
  Then we take any inner partition $\underline{\rho_j}$ of $\rho_j$ for $j = 1, \dots, l$.
  Note that the Kreweras complement $\K(\rho_j)$ equals $\K(\rho)$ restricted to the interval $I_j$.
  Hence, since $\underline{\rho_j}$ is at a Kreweras point of the outer partition $\overline{\rho_j}=\rho_j \setminus\underline{\rho_j}$,
  $\underline{\rho_j}$ is also at a Kreweras point of $\rho \setminus \underline{\rho_j}$.

  Conversely, we take any inner partition $\underline{\rho}$ of $\rho$.  By the definition of inner partitions, $\underline{\rho}$ is supported on some interval $I_j$.
  If $\underline{\rho}$ contains the first block of $\rho_j$, then $\underline{\rho}$ equals $\rho_j$.
 Otherwise, the support of $\underline{\rho}$ is a sub-interval of $I_j$ which does not intersect the first block of $\rho_j$, and since $\underline{\rho}$ is at a Kreweras point of the outer partition $\overline{\rho}$, $\underline{\rho}$ is also at a Kreweras point of $\rho_j \setminus \underline{\rho}$.
\end{proof}

\begin{Prop} \label{prop:Kreweras_decomposition} Let $k \ge 2$.
 Each $\rho \in \NC(k)$ has exactly $\abs{\rho}-1$ Kreweras decompositions.
\end{Prop}

\begin{proof}
The proof runs  by induction.
  It is clear that the statement is true when $k=2$.
  Then we assume the statement is true up to $k-1$ and take $\rho \in \NC(k) \; (\abs{\rho}>1)$.
  Suppose that the first block of $\rho$ divides $[k]$ into $l$ segments $I_1, \dots, I_l$.
  Then all $\{ \rho_j = \rho \! \mid_{I_j} \}_{j=1}^{l}$ are inner partitions of $\rho$.
  By Lemma \ref{lem:inner_partition}, a subset of $\rho$ is an inner partition of $\rho$ if and only if
  it is one of $\{ \rho_j \}_{j=1}^{l}$ or an inner partition of some $\rho_j$.
  Therefore, by the induction hypothesis,
  the number of inner partitions of $\rho$ is $l + \sum_{j =1}^{l} (\abs{\rho_j} - 1) = \abs{\rho} - 1$.
\end{proof}

\begin{Exam}
We take $\rho \in \NC(27)$ to be
  
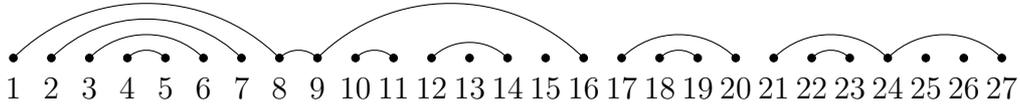
\[
    \begin{tikzpicture}[scale=0.5]
      \foreach \x in {1,...,27}{
         \draw[circle,fill] (\x,0)circle[radius=1mm];
          \node at (\x,-0.8) {$\x$};
      }
      \foreach \x/\y in {1/8,8/9,9/16,2/7,3/6,4/5,10/11,12/14,17/20,18/19,21/24,24/27,22/23} {
         \draw(\x,0) to[bend left=45] (\y,0);
      }
    \end{tikzpicture}
  \]
in which $\abs{\rho} = 14$.
  The three non-crossing partitions 
    
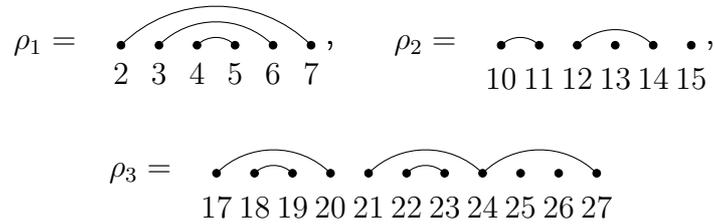
\[
      \begin{tikzpicture}[scale=0.5]
      \node at (0,0) {$\rho_1=$};
        \foreach \x in {2,...,7}{
           \draw[circle,fill] (\x,0)circle[radius=1mm]node[below]{$\x$};
        }
        \foreach \x/\y in {2/7,3/6,4/5} {
           \draw(\x,0) to[bend left=45] (\y,0);
        \node at (7.5,0) {,};
        }
      \node at (10,0) {$\rho_2=$};
        \foreach \x in {10,...,15}{
           \draw[shift={(2,0)},circle,fill] (\x,0)circle[radius=1mm]node[below]{$\x$};
        }
        \foreach \x/\y in {10/11,12/14} {
           \draw[shift={(2,0)}](\x,0) to[bend left=45] (\y,0);
           \node at (17.5,0) {,};
        }
      \end{tikzpicture}
    \]
    \[
      \begin{tikzpicture}[scale=0.5]
       \node at (15,0) {$\rho_3=$};
        \foreach \x in {17,...,27}{
           \draw[circle,fill] (\x,0)circle[radius=1mm]node[below]{$\x$};
        }
        \foreach \x/\y in {17/20,18/19,21/24,24/27,22/23} {
           \draw(\x,0) to[bend left=45] (\y,0);
        }
      \end{tikzpicture}
    \]
are also inner partitions of $\rho$ and the two inner partitions of $\rho_1$
  
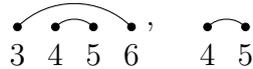
\[
    \begin{tikzpicture}[scale=0.5]
      \foreach \x in {3,...,6}{
         \draw[circle,fill] (\x,0)circle[radius=1mm]node[below]{$\x$};
      }
      \foreach \x/\y in {3/6,4/5} {
         \draw(\x,0) to[bend left=45] (\y,0);
      }
       \node at (6.5,0) {,};
      \foreach \x in {4,5}{
         \draw[shift={(4,0)},circle,fill] (\x,0)circle[radius=1mm]node[below]{$\x$};
      }
      \foreach \x/\y in {4/5} {
         \draw[shift={(4,0)}](\x,0) to[bend left=45] (\y,0);
      }
    \end{tikzpicture}
  \]
  are inner partitions of $\rho$.
  In the same way, the three inner partitions of $\rho_2$
  
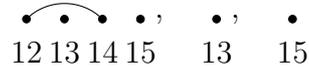
\[
    \begin{tikzpicture}[scale=0.5]
      \foreach \x in {12,...,15}{
         \draw[circle,fill] (\x,0)circle[radius=1mm]node[below]{$\x$};
      }
      \foreach \x/\y in {12/14} {
         \draw(\x,0) to[bend left=45] (\y,0);
      }
    \node at (15.5,0) {,};
      \foreach \x in {13}{
         \draw[shift={(4,0)},circle,fill] (\x,0)circle[radius=1mm]node[below]{$\x$};
      }
       \node at (17.5,0) {,};
      \foreach \x in {15}{
         \draw[shift={(4,0)},circle,fill] (\x,0)circle[radius=1mm]node[below]{$\x$};
      }
    \end{tikzpicture}
  \]
  and the five inner partitions of $\rho_3$
  
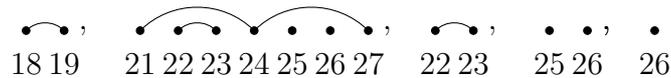
\[
    \begin{tikzpicture}[scale=0.5]
      \foreach \x in {18,19}{
         \draw[circle,fill] (\x,0)circle[radius=1mm]node[below]{$\x$};
      }
       \node at (19.5,0) {,};
      \foreach \x/\y in {18/19} {
         \draw(\x,0) to[bend left=45] (\y,0);
      }
  \node at (27.5,0) {,};
       \foreach \x in {21,...,27}{
         \draw[circle,fill] (\x,0)circle[radius=1mm]node[below]{$\x$};
      }
      \foreach \x/\y in {21/24,24/27,22/23} {
         \draw(\x,0) to[bend left=45] (\y,0);
      }
    \end{tikzpicture}
    \begin{tikzpicture}[scale=0.5]
      \foreach \x in {22,23}{
         \draw[circle,fill] (\x,0)circle[radius=1mm]node[below]{$\x$};
      }
      \foreach \x/\y in {22/23} {
         \draw(\x,0) to[bend left=45] (\y,0);
         \node at (23.5,0) {,};
      }
    \foreach \x in {25,26}{
         \draw[circle,fill] (\x,0)circle[radius=1mm]node[below]{$\x$};
         \node at (26.5,0) {,};
      }
    \end{tikzpicture}
    \begin{tikzpicture}[scale=0.5]
      \foreach \x in {26}{
         \draw[circle,fill] (\x,0)circle[radius=1mm]node[below]{$\x$};
      }
    \end{tikzpicture}
  \]
  are also inner partitions of $\rho$.  Thus $\rho$ has 13 inner partitions: $\rho_1,\rho_2,\rho_3$ and the inner partitions of them.
\end{Exam}

\subsection{Proof of Theorem \ref{main}} \label{subsec:main}
In this subsection, we follow the notation in Theorem \ref{main}. The index $N$ is omitted for readability when no confusion occurs.
The main part of the proof of Theorem \ref{main} is the following. 

\begin{Thm} \label{thm:main1.1} 
Assume that \begin{equation}\label{eq:ass2}
\sup_{N\ge1} \E[\M_k(\m_N)] <\infty, \qquad \forall k \in 2\N.   
\end{equation}
Then, for every $k \in \N$ and $\ell \in\{1,2\}$, it holds that
  \[
    \E[\M_k(\kappa_N)^\ell] = \E[\M_k(\tau_N)^\ell] + O \left( \frac{1}{N} \right). 
  \]
\end{Thm}
\begin{Rem}
Whether the above result holds for $\ell \ge3$ is unknown. 
\end{Rem}

\begin{proof} Note first that the assumption \eqref{eq:ass2} implies that 
\begin{equation}
\sup_{N\ge1} \E[\abs{\M_\sigma(\m_N)}] <\infty, \qquad \sigma \in \S_k  
\end{equation}
for every $k \in \N$, thanks to the iterative use of Schwarz inequality and \eqref{eq:hoelder}.

\vspace{3mm}
\noindent
{\bf (i) $\ell=1$.} A key of the proof is the calculations of
\begin{equation} \label{eq:start0}
  \sum_{j=1}^{N-1}\E[\tilde\lambda_j{}^k]=  \E \circ \Tr[(PU^*DUP)^k],
\end{equation}
where $P = \diag(1, \dots, 1, 0)$.
The RHS of (\ref{eq:start0}) is calculated into
\begin{align}
      \E \circ \Tr[(PU^*DUP)^k] &= \E \circ \Tr[(DUPU^*)^k] \notag \\
      &= \sum_{\sigma, \pi \in \S_k} \E \circ \Tr_\sigma[D, \dots, D] \Tr_\pi[P, \dots, P]\Wg(\pi^{-1}\sigma^{-1}\gamma_k) \label{eq:start}\\
     &= \sum_{\sigma, \pi \in \S_k} \E \circ \Tr_\sigma[D, \dots, D] \Tr_{\pi^{-1}\gamma_k}[P, \dots, P]\Wg(\sigma^{-1}\pi) \notag \\
    &= \sum_{\sigma, \pi \in \S_k} N^{\# (\sigma)} \E \circ \tr_\sigma[D, \dots, D] (N-1)^{\# (\pi^{-1}\gamma_k)}\Wg(\sigma^{-1}\pi),
      \label{eq:WG1}
\end{align}
where \eqref{eq:WG_formula} was used on the second line and the change of variables $\pi\mapsto \pi^{-1}\gamma_k$ was employed on the third line.

On the other hand, if the projection $P$ is replaced by the identity $I$ in (\ref{eq:start}), then the same calculations lead to
\begin{align}
    \sum_{i=1}^{N} \E[{\lambda_i}^k] &= \E \circ \Tr[D^k] = \E \circ \Tr[(DUIU^*)^k]  \notag \\
    &= \sum_{\sigma, \pi \in \S_k} N^{\# (\sigma)} \E \circ \tr_\sigma[D, \dots, D] N^{\# (\pi^{-1}\gamma_k)}\Wg(\sigma^{-1}\pi). \label{eq:WG2}
\end{align}
Taking the difference of \eqref{eq:WG1} and \eqref{eq:WG2} provides
\begin{equation*}
  \begin{split}
    \E [\M_k(\kappa_N)]
    &= \sum_{i=1}^{N} \E[{\lambda_i}^k] - \sum_{j=1}^{N-1}\E[\tilde\lambda_j{}^k] \\
    &= \sum_{\sigma, \pi \in \S_k} N^{\# (\sigma)} \E \circ \tr_\sigma[D, \dots, D]
    \#(\pi^{-1}\gamma_k)N^{\#(\pi^{-1}\gamma_k) -1} \left( 1+ O(N^{-1})\right)\Wg(\sigma^{-1}\pi).
  \end{split}
\end{equation*}
Here we use the asymptotic expansion \eqref{eq:asymptotic_Weingarten} of the Weingarten functions to get
\begin{equation}\label{eq:tau1}
    \E [\M_k(\kappa_N)] = \sum_{\abs{\sigma} + \abs{\sigma^{-1}\pi} + \abs{\pi^{-1}\gamma_k} = k-1}
    \#(\pi^{-1}\gamma_k) \E[ \M_\sigma(\m_N)] \mu_k(\sigma^{-1}\pi)
    + O \left( \frac{1}{N} \right).
\end{equation}
Using the isomorphism explained in Section \ref{subsec:Isomorph}, we may rewrite \eqref{eq:tau1} in terms of non-crossing partitions: 
\begin{align}
   \E [\M_k(\kappa_N)]
    &= \sum_{ \nu \le \rho \in \NC(k)} \abs{\K(\rho)}\, \E[\M_\nu(\m_N)] \ncmu_k(\nu, \rho) +O \left( \frac{1}{N} \right)  \notag \\
    &= \sum_{ \rho \in \NC(k)} \abs{\K(\rho)}\, \E[\fc_\rho(\m_N)] +O \left( \frac{1}{N} \right),  \label{eq:tau2}
\end{align}
where the cumulant--moment formula \eqref{eq:cm} was used in the last line.
Combining \eqref{eq:tau2} and Theorem \ref{main2} implies the desired conclusion.




\vspace{3mm}
\noindent
{\bf (ii) $\ell=2$.}  Taking the expectation of $\M_k(\kappa_N)^2 = (\Tr[D^k] - \Tr[(DUPU^*)^k])^2$ with Weingarten calculus yields
\begin{align}
  &\E[\M_k(\kappa_N)^2] \notag \\
  &= \E \circ \Tr_{{\gamma_k}^{(1)}{\gamma_k}^{(2)}}[(DUIU^{*})^{k},(DUIU^{*})^{k}] - \E \circ \Tr_{{\gamma_k}^{(1)}{\gamma_k}^{(2)}}[(DUIU^{*})^k, (DUPU^{*})^{k}]  \notag \\
  &\quad - \E \circ \Tr_{{\gamma_k}^{(1)}{\gamma_k}^{(2)}}[(DUPU^{*})^{k},(DUIU^{*})^{k}] + \E \circ \Tr_{{\gamma_k}^{(1)}{\gamma_k}^{(2)}}[(DUPU^{*})^{k},(DUPU^{*})^{k}] \notag \\
  &=\sum_{\sigma,\pi \in \S_{2k}} N^{\#(\sigma)} \E[ \M_\sigma(\m_N) ]  \mathcal{T}_k(\pi)\Wg(\sigma^{-1}\pi) \label{eq:second_moment}
\end{align}
where
\[
  \mathcal{T}_k(\pi) := \Tr_{\pi^{-1}{\gamma_k}^{(1)}{\gamma_k}^{(2)}}[I^{k},I^{k}] - \Tr_{\pi^{-1}{\gamma_k}^{(1)}{\gamma_k}^{(2)}}[I^{k},P^{k}] - \Tr_{\pi^{-1}{\gamma_k}^{(1)}{\gamma_k}^{(2)}}[P^{k},I^{k}] + \Tr_{\pi^{-1}{\gamma_k}^{(1)}{\gamma_k}^{(2)}}[P^{k},P^{k}].
\]
Note that, for readability,  we use the abbreviation $\Tr_\sigma[A^k,B^k] = \Tr_\sigma[A_1, \dots, A_k, B_1, \dots, B_k]$ when $A = A_1 = \cdots = A_k$ and $B = B_1 = \cdots = B_k$.

By using the evident decomposition $I = P + Q$ with $Q = \diag(0, \dots, 0, 1)$,
we have the following expansion
\begin{align}
  &\Tr_{\pi^{-1}{\gamma_k}^{(1)}{\gamma_k}^{(2)}}[I^{k},I^{k}] \notag \\
  &= \Tr_{\pi^{-1}{\gamma_k}^{(1)}{\gamma_k}^{(2)}}[I^{k},I^{k-1},P+Q] \notag \\
  &=\Tr_{\pi^{-1}{\gamma_k}^{(1)}{\gamma_k}^{(2)}}[I^{k},I^{k-1},P] + \Tr_{\pi^{-1}{\gamma_k}^{(1)}{\gamma_k}^{(2)}}[I^{k},I^{k-1},Q] \notag \\
  &=\Tr_{\pi^{-1}{\gamma_k}^{(1)}{\gamma_k}^{(2)}}[I^{k},I^{k-2},P+Q,P] + \Tr_{\pi^{-1}{\gamma_k}^{(1)}{\gamma_k}^{(2)}}[I^{k},I^{k-1},Q] \notag \\
  &=\Tr_{\pi^{-1}{\gamma_k}^{(1)}{\gamma_k}^{(2)}}[I^{k},I^{k-2},P,P] + \Tr_{\pi^{-1}{\gamma_k}^{(1)}{\gamma_k}^{(2)}}[I^{k},I^{k-2},Q,P] + \Tr_{\pi^{-1}{\gamma_k}^{(1)}{\gamma_k}^{(2)}}[I^{k},I^{k-1},Q] \notag \\
  & \cdots \notag \\
  &=\Tr_{\pi^{-1}{\gamma_k}^{(1)}{\gamma_k}^{(2)}}[I^k,P^{k}] + \sum_{j=1}^{k} \Tr_{\pi^{-1}{\gamma_k}^{(1)}{\gamma_k}^{(2)}}[I^{k},I^{k-j},Q,P^{j-1}]. \label{eq:Tr1}
\end{align}
In the same way,
\begin{equation}
  \Tr_{\pi^{-1}{\gamma_k}^{(1)}{\gamma_k}^{(2)}}[P^{k},I^{k}] =
  \Tr_{\pi^{-1}{\gamma_k}^{(1)}{\gamma_k}^{(2)}}[P^k,P^{k}] + \sum_{j=1}^{k} \Tr_{\pi^{-1}{\gamma_k}^{(1)}{\gamma_k}^{(2)}}[P^{k},I^{k-j},Q,P^{j-1}].\label{eq:Tr2}
\end{equation}
Combining \eqref{eq:Tr1} and \eqref{eq:Tr2} together we get
\[
\mathcal{T}_k(\pi) = \sum_{j=1}^{k} \left(\Tr_{\pi^{-1}{\gamma_k}^{(1)}{\gamma_k}^{(2)}}[I^{k},I^{k-j},Q,P^{j-1}] - \Tr_{\pi^{-1}{\gamma_k}^{(1)}{\gamma_k}^{(2)}}[P^{k},I^{k-j},Q,P^{j-1}]\right).
\]
Again, a similar argument yields
\[
\mathcal{T}_k(\pi)=\sum_{i=1}^{k} \sum_{j=1}^{k} \Tr_{\pi^{-1}{\gamma_k}^{(1)}{\gamma_k}^{(2)}}[I^{k-i},Q,P^{i-1},I^{k-j},Q,P^{j-1}].
\]

When we decompose $\pi^{-1}{\gamma_k}^{(1)}{\gamma_k}^{(2)}$ into cycles,
the contribution of the cycle which contains $Q$ is at most $1$ in $\Tr_{\pi^{-1}{\gamma_k}^{(1)}{\gamma_k}^{(2)}}[I^{k-i},Q,P^{i-1},I^{k-j},Q,P^{j-1}]$.
Therefore, we get the upper bound
\begin{equation}\label{eq:upper_bound}
\mathcal{T}_k(\pi) = O \left( N^{\#(\pi^{-1}{\gamma_k}^{(1)}{\gamma_k}^{(2)})-1} \right) = O \left( N^{2k -\abs{\pi^{-1}{\gamma_k}^{(1)}{\gamma_k}^{(2)}}-1} \right).
\end{equation}
Here we also use the asymptotic expansion \eqref{eq:asymptotic_Weingarten} of the Weingarten functions
and two elementary facts about the length functions in symmetric groups $\S_{2k}$:
$\abs{\sigma} + \abs{\sigma^{-1}\pi} + \abs{\pi^{-1}{\gamma_k}^{(1)}{\gamma_k}^{(2)}} \ge 2k-2$ and
$\abs{\sigma} + \abs{\sigma^{-1}\pi} + \abs{\pi^{-1}{\gamma_k}^{(1)}{\gamma_k}^{(2)}} \neq 2k-1$
since $\abs{\sigma} + \abs{\sigma^{-1}\pi} + \abs{\pi^{-1}{\gamma_k}^{(1)}{\gamma_k}^{(2)}} \equiv \abs{{\gamma_k}^{(1)}{\gamma_k}^{(2)}} \equiv 2k-2 \pmod 2$ by using the length property (\ref{eq:mod2}).
Applying those facts and \eqref{eq:upper_bound} to \eqref{eq:second_moment} reveals that
\[
  \E[\M_k(\kappa_N)^2]
  = \sum_{\substack{\sigma,\pi\in\S_{2k} \\ \abs{\sigma} + \abs{\sigma^{-1}\pi} + \abs{\pi^{-1}{\gamma_k}^{(1)}{\gamma_k}^{(2)}} = 2k-2}} N^{- \abs{\sigma} - \abs{\sigma^{-1}\pi}}\E[ \M_\sigma(\m_N) ]  \mathcal{T}_k(\pi)\mu_{2k}(\sigma^{-1}\pi) + O \left( \frac{1}{N} \right).
\]
By using the isomorphism (\ref{isomoph_part}), the last expression can be rewritten in terms of non-crossing partitions: 
\[
\begin{split}
   \E[\M_k(\kappa_N)^2] &= \sum_{\substack{\nu,\rho\in \NC(2k) \\ \nu \le \rho \le ({1_k}^{(1)},{1_k}^{(2)})}} \E[ \M_\nu(\m_N)] \ncmu_{2k}(\nu,\rho) \frac{\mathcal{T}_k(\mathcal{P}_\rho)}{N^{\abs{\K(\rho)}-1}} + O \left( \frac{1}{N} \right) \\
   &= \sum_{\substack{\rho\in \NC(2k) \\ \rho \le ({1_k}^{(1)},{1_k}^{(2)})}} \E[ \fc_\rho(\m_N)] \frac{\mathcal{T}_k(\mathcal{P}_\rho)}{N^{\abs{\K(\rho)}-1}} + O \left( \frac{1}{N} \right) \\
  &= \sum_{\rho_1,\rho_2 \in \NC(k)} \E[ \fc_{\rho_1}(\m_N)\fc_{\rho_2}(\m_N)] \frac{\mathcal{T}_k(\mathcal{P}_{(\rho_1,\rho_2)})}{N^{\abs{\K(\rho_1)}+\abs{\K(\rho_2)}-2}} + O \left( \frac{1}{N} \right).
\end{split}
\]
Note that
\begin{align}
  \mathcal{T}_k(\mathcal{P}_{(\rho_1,\rho_2)}) &= \sum_{i=1}^{k} \sum_{j=1}^{k} \Tr_{(\K(\rho_1),\K(\rho_2))}[I^{k-i},Q,P^{i-1},I^{k-j},Q,P^{j-1}]  \notag \\
  &= N^{\abs{\K(\rho_1)}+\abs{\K(\rho_2)}-2} \abs{\K(\rho_1)} \abs{\K(\rho_2)} + O \left( N^{\abs{\K(\rho_1)}+\abs{\K(\rho_2)}-3} \right).   \label{eq:Tpi}
\end{align}
This is because the contribution of a cycle is $0$ if it contains both $P$ and $Q$, and is 1 if it contains $Q$ and no $P$; from those observations, the main contributions appear when both $Q$'s are at Kreweras points of $\rho_1$ and $\rho_2$, respectively, and so \eqref{eq:Tpi} follows. 
Hence we arrive at the formula
\begin{equation} \label{eq:expectation_square}
  \E[\M_k(\kappa_N)^2] = \sum_{\rho_1,\rho_2 \in \NC(k)} \abs{\K(\rho_1)} \abs{\K(\rho_2)} \,\E[ \fc_{\rho_1}(\m_N)\!\fc_{\rho_2}(\m_N)]
  + O \left( \frac{1}{N} \right).
\end{equation}
Applying Theorem \ref{main2} to the RHS finishes the proof. 
\end{proof}

\begin{Rem}
Note that the calculations for $\ell=1$ are similar to those in \cite[pp.379-393]{Nica-Speicher} where asymptotic freeness is proved for matrices rotated by independent Haar unitaries.
\end{Rem}

\begin{proof}[Proof of Theorem \ref{main}] According to Theorem \ref{main2}, $\M_k(\tau_N)^\ell$ is a polynomial on $\{\M_n(\m_N)\}_{n\in\N}$, so that Proposition \ref{prop:polynomial} allows us to pass to the limit: 
\begin{equation}
\lim_{N\to\infty}\E[\M_k(\tau_N)]  = \M_k(\tau) \quad \text{and} \quad \lim_{N\to\infty}\E[\M_k(\tau_N)^2]  = \M_k(\tau)^2, \qquad k \in \N. 
\end{equation}
Combining the above and Theorem \ref{thm:main1.1} yields that 
\begin{equation}
\lim_{N\to\infty}\E[\M_k(\kappa_N)]  = \M_k(\tau) \quad \text{and} \quad \lim_{N\to\infty}\E[\M_k(\kappa_N)^2]  = \M_k(\tau)^2, \qquad k \in \N,  
\end{equation}
which readily implies $\|\M_k(\kappa_N) - \M_k(\tau)\|_{L^2} \to 0$. In particular, $\M_k(\kappa_N)$ converges to $\M_k(\tau)$ in probability for every $k\in \N$. Since $\M_k(\gg_N)$ and $\M_k(\m)$ are respectively expressed by a common polynomial evaluated at $\{\M_k(\kappa_N)\}_{k\ge1}$ and $\{\M_k(\tau)\}_{k\ge1}$, it follows that $\M_k(\gg_N)$ converges to $\M_k(\m)$ in probability. Finally, if the moment problem for $\{\M_k(\m)\}_{k\ge1}$ is determinate then we conclude that $\gg_N$ weakly converges to $\m$ in probability by Proposition \ref{prop:weak_prob}. 
\end{proof}


\appendix 

\section{Appendix}\label{App1}
Some results on the moment method for random measures are collected below. The proofs are basic. 
Let $\pp, \pp_n, n\in \N,$ be random probability measures on $\R$ with an underlying probability space $(\Omega,\mathcal F, \mathbb P)$ below. 

\begin{Prop}\label{prop:weak_prob}
Suppose that $\pp_n,\pp, n\in \N$ have finite moments of all orders almost surely, and the moment problem for $\{\M_k(\pp)\}_{k\ge1}$ is determinate almost surely. If 
\begin{equation}\label{eq:weak_moment}
\lim_{n\to\infty}\mathbb P [\abs{\M_k(\pp_n) - \M_k(\pp)} \ge \epsilon] =0, \qquad  k \in \N,~  \epsilon>0, 
\end{equation}
then $\pp_n$ weakly converges to $\pp$ in probability: 
\begin{equation}\label{eq:weak_prob}
\lim_{n\to\infty}\mathbb P \left[\left|\int_\R f(x) \,d\pp_n(x) - \int_\R f(x) \,d\pp(x)\right| \ge \epsilon \right] =0, \qquad  f \in C_b(\R), ~ \epsilon>0. 
\end{equation}
\end{Prop}
\begin{proof} 
For later use, we first verify the existence of a subsequence of $\{\pp_n\}_{n\ge1}$ which weakly converges to $\pp$ almost surely. 
Let $\Omega_0 \in\mathcal{F}$ be such that $\mathbb P[\Omega_0]=1$ and the moment problem for $\{\M_k(\pp^\omega)\}_{k\ge1}$ is determinate for all $\omega \in \Omega_0$. 
For $k=1$, there exists a subsequence $\{n(1,\ell)\}_{\ell=1}^\infty$ of $\N$ and $\Omega_1 \subset \Omega_0$ such that $\Omega_1 \in \mathcal F, \mathbb P[\Omega_1]=1$ and $\M_1(\pp_{n(1,\ell)}^\omega)$ converges to $\M_1(\pp^\omega)$ for all $\omega \in \Omega_1$. For $k=2,$ there exists a subsequence $\{n(2,\ell)\}_{\ell=1}^\infty$ of $\{n(1,\ell)\}_{\ell=1}^\infty$ and $\Omega_2 \subset \Omega_1$ such that $\Omega_2 \in \mathcal F, \mathbb P[\Omega_2]=1$ and $\M_2(\pp_{n(2,\ell)}^\omega)$ converges to $\M_2(\pp^\omega)$ for all $\omega \in \Omega_2$. In this way we obtain subsequences $\{n(k,\ell)\}_{\ell=1}^\infty$ and decreasing subsets $\Omega_k$ of probability one for $k \ge1$. Define $\tilde\Omega:= \cap_{k\ge1}\Omega_k$ and $n(\ell):=n(\ell,\ell)$; then $\M_k(\pp_{n(\ell)}^\omega)$ converges  to $\M_k(\pp^\omega)$ as $\ell \to \infty$ for all $\omega \in \tilde \Omega$ and all $k \in \N$. Since the moment problem for the limit sequence is determinate, we conclude by \cite[Theorem 4.5.5]{Chung} that $\pp_{n(\ell)}^\omega$ weakly converges to $\pp^\omega$ as $\ell \to \infty$ for all $\omega \in \tilde \Omega$. 

To finish the proof, suppose to the contrary that the desired conclusion \eqref{eq:weak_prob} is false: there exist $f \in C_b(\R)$, $\epsilon,\delta>0$ and a subsequence of $\{\pp_n\}_{n\ge1}$, denoted by $\{\pp_{n'}\}$, such that 
\begin{equation}\label{eq:weak_prob2}
\mathbb P \left[\left|\int_\R f(x) \,d\pp_{n'}(x) - \int_\R f(x) \,d\pp(x)\right| \ge \epsilon \right] \ge \delta, \qquad \forall n'.  
\end{equation}
However, we can extract a further subsequence of $\{\pp_{n'}\}$ which weakly converges to $\pp$ almost surely as we discussed. For this subsequence, the LHS of \eqref{eq:weak_prob2} must tend to zero, a contradiction. 
\end{proof}

\begin{Rem} 
A similar result and proof are found in \cite[p.\ 178--180]{Gre}. 
\end{Rem}

\begin{Prop}\label{prop:Lp}
Suppose that 
\begin{equation}\label{eq:ass}
\sup_{n\ge1} \E[\M_k(\pp_n)] <\infty \quad \text{and}\quad\E[\M_k(\pp)]<\infty, \qquad  k \in 2\N. 
\end{equation}
Then the condition \eqref{eq:weak_moment} is equivalent to 
\begin{equation}\label{eq:Lp_moment}
\|\M_k(\pp_n) - \M_k(\pp)\|_{L^p} \to 0, \qquad  p \in [1,\infty), ~ k \in \N. 
\end{equation}
\end{Prop}
\begin{proof}
It suffices to prove that  \eqref{eq:weak_moment} implies \eqref{eq:Lp_moment}; the other direction is well known. 

For $p \in [1,\infty)$ choose $\ell \in 2\N$ such that $\ell \ge p$. The H\"older inequality implies that $|\M_k(\pp_n)|^{\ell} \le \M_{k\ell}(\pp_n)$ and hence 
\begin{equation}\label{eq:hoelder}
\|\M_k(\pp_n)\|_{L^p} \le \|\M_k(\pp_n)\|_{L^\ell} \le (\E[\M_{k\ell}(\pp_n)])^{\frac1{\ell}}.  
\end{equation}
Combining the above and \eqref{eq:ass}, as well as similar inequalities for $\M_k(\pp)$, yields that 
\begin{equation}\label{eq:uniform}
\sup_{n\in\N}\|\M_k(\pp_n)\|_{L^p}<\infty\quad \text{and} \quad \|\M_k(\pp)\|_{L^p}<\infty,  \qquad k\in \N,~p \in [1,\infty).   
\end{equation}
By standard arguments we obtain  
\begin{align*}
\|\M_k(\pp_n) - \M_k(\pp)\|_{L^p}^p 
&= \E[\abs{\M_k(\pp_n) - \M_k(\pp)}^p 1_{\{\abs{\M_k(\pp_n) - \M_k(\pp)} \ge \epsilon\}}] \\
&\quad + \E[\abs{\M_k(\pp_n) - \M_k(\pp)}^p 1_{\{\abs{\M_k(\pp_n) - \M_k(\pp)} < \epsilon\}}] \\
& \le (\E[\abs{\M_k(\pp_n) - \M_k(\pp)}^{2p}])^{1/2} (\mathbb P[\abs{\M_k(\pp_n) - \M_k(\pp)} \ge \epsilon])^{1/2} + \epsilon^p \\
&\le  (\| \M_k(\pp_n)\|_{L^{2p}} + \| \M_k(\pp)\|_{L^{2p}})^p  (\mathbb P [\abs{\M_k(\pp_n) - \M_k(\pp)} \ge \epsilon])^{1/2} + \epsilon^p.  
\end{align*}  
Applying \eqref{eq:uniform} and \eqref{eq:weak_moment} to the above finishes the proof. 
\end{proof}

\begin{Prop}\label{prop:polynomial} Suppose that \eqref{eq:weak_moment} and \eqref{eq:ass} hold. Then 
\begin{equation*}
\lim_{n\to\infty}\E[P(\M_1(\pp_n), \M_2(\pp_n), \dots, \M_k(\pp_n))] = \E[P(\M_1(\pp), \M_2(\pp), \dots, \M_k(\pp))] 
\end{equation*}
for every $k \in \N$ and every polynomial $P \in \C[x_1,x_2,\dots,x_k]$. 
\end{Prop}
\begin{proof} This is a consequence of Proposition \ref{prop:Lp} and the following standard fact: if random variables $Y,Z, Y_n,Z_n,  n\in \N$ satisfy $Y_n\to Y $ in $L^p$ and $Z_n \to Z$ in $L^p$ for all $p \in[1,\infty)$, then $Y_n Z_n \to YZ$ in $L^p$ for all $p \in [1,\infty)$. 
\end{proof}

\section*{Acknowledgments}
T.H.\ is supported by JSPS Grant-in-Aid for Young Scientists 19K14546 and 18H01115. This work was supported by JSPS Open Partnership Joint Research Projects grant no.\ JPJSBP120209921 and Bilateral Joint Research Projects (JSPS-MEAE-MESRI, grant no.\ JPJSBP120203202).  The authors express sincere thanks to Sho Matsumoto for pointing out a proof of Theorem \ref{main2} based on \cite{Lassalle} as mentioned in the subsequent paragraph.


\begin{flushleft}

Department of Mathematics, Hokkaido University, North 10 West 8, Kita-Ku, Sapporo 060-0810, Japan

email: kfujie@eis.hokudai.ac.jp

\vspace{5mm}

Department of Mathematics, Hokkaido University, North 10 West 8, Kita-Ku, Sapporo 060-0810, Japan

email: thasebe@math.sci.hokudai.ac.jp

\end{flushleft}

\end{document}